\documentclass[12pt]{article}

\usepackage{url}
\usepackage{amsmath,amsthm,amssymb,amsfonts,graphicx,enumerate}
\usepackage{tikz}
\usepackage{float}
\usepackage{caption}
\usepackage{subcaption}
\usepackage{slashbox}
\usepackage{rotating}

\newtheorem{theorem}{Theorem}[section]

\newtheorem{question}[theorem]{Question}
\newtheorem{lemma}[theorem]{Lemma}
\newtheorem{corollary}[theorem]{Corollary}
\newtheorem{proposition}[theorem]{Proposition}
\theoremstyle{definition}
\newtheorem{definition}[theorem]{Definition}
\newtheorem{example}[theorem]{Example}

\theoremstyle{remark}

\newcommand{\rank}{\textup{rank}\,}

\DeclareMathOperator{\Tr}{Tr}

\newcommand{\RR}{\mathbb R}

\title{The phase rank of a matrix}

\author{Ant\'onio Pedro Goucha \and Jo\~{a}o Gouveia  \thanks {
The first author's research was supported through a PhD scholarship from FCT, grant PD/BD/135276/2017.  The second author was supported by the Centre for
Mathematics of the University of Coimbra, grant UIDB/00324/2020, funded by the Portuguese
Government through FCT/MCTES and form the grant P2020 SAICTPAC/0011/2015.}}

\date{CMUC, Department of Mathematics, University of Coimbra}

\begin{document}

\maketitle

\begin{abstract}
In this paper, we introduce and study the notion of phase rank. Given a matrix filled with phases, i.e., with complex entries of modulus $1$, its phase rank is the smallest possible rank for a complex matrix with the same phase, but possible different modulus. This notion generalizes the notion of sign rank, and is the complementary notion to that of phaseless rank. It also has intimate connections to the problem of ray nonsingularity and provides an important class of examples for the study of coamoebas.
Among other results, we provide a new characterization of phase rank $2$ for $3 \times 3$ matrices, that implies the new result that for the $3\times 3$ determinantal variety, its coamoeba is determined solely by the colopsidedness criterion.
\end{abstract}

\section{Introduction}\label{sec:mot}

In this paper, we study the problem of minimizing the rank of a complex matrix given the phase of its entries, but not their modulus. This is a natural offshoot of the previous work  \cite{GG21}, where we studied the  complementary problem of minimizing the rank of a complex matrix when we are given the modulus of the entries but not their phases. While similar in formulation and with obvious links, these problems are of very different natures and need distinct approaches.

To further specify our scope, we will be interested in \emph{phase matrices}, that is to say, matrices whose entries are complex numbers of modulus one. We now want to minimize the rank among all complex matrices who have this entrywise phases. More precisely, we are interested in the following quantity.

\begin{definition}[\textbf{Phase rank}]
Let $S^1=\{z \in \mathbb{C}: |z|=1\}$ and $\Theta$ be a phase matrix, i.e., with entries in $S^1$. We define its phase rank as
$$\rank_{\text{phase}}(\Theta) = \min \{\rank(M): M \in (\mathbb{C}^{*})^{n\times m}, \frac{M_{ij}}{|M_{ij}|}= \Theta_{ij}, \forall i,j\}.$$
\end{definition}

Here, $\mathbb{C}^{*}$ is the multiplicative group of all complex numbers except zero.
An alternative way of defining this quantity is
$$\rank_{\text{phase}}(\Theta) = \min \{\rank(M*\Theta): M\text{ is a matrix with positive entries }\},$$
where  $*$ denotes the Hadamard product of matrices.

\begin{example}\label{ex:phase_rank}
Let $$\Theta=\begin{bmatrix}
   1 & 1 & 1\\
   1 & i & -i\\
   1 & 1 & i
\end{bmatrix}
\text{ and }M=\begin{bmatrix}
   1 & 1 & 1\\
   1 & \sqrt{3}i & -\sqrt{3}i\\
   1 & 3 & \sqrt{3}i
\end{bmatrix}.$$
Since $\rank(M)=2, \rank_{\text{phase}}(\Theta)\leq 2.$
\end{example}

By definition, the matrices that we are considering in our minimization problem have nonzero entries, since zero does not have a well-defined phase. On could consider \emph{generalized phase matrices}, whose entries are allowed to be either complex numbers of modulus one or zero. We can generalize the notion of phase rank to this larger class by saying that for a generalized phase matrix $\bar{\Theta}$ its phase rank is the smallest rank of a matrix with the same support as $\bar{\Theta}$ and whose nonzero entries have the same phase as the corresponding entries of $\bar{\Theta}$.
In this paper, however, we will restrict ourselves to phase ranks of proper phase matrices.

There is a natural version of phase rank on the reals, the sign rank. The history on that topic goes back to work in the 1980s in communication complexity \cite{alon1985geometrical, paturi1986probabilistic}, but has antecedents in older pioneer work on the problem of \emph{signsolvability}: when is every matrix with a given generalized sign pattern guaranteed to be invertible, a problem that goes back to the 1960s (see for  instance \cite{klee1984signsolvability,maybe1969qualitative}). This whole area of research has seen a renewed interest in the last few years in both the mathematical and the computer science communities and several important developments have appeared \cite{MR2170271,razborov2010sign,bhangale2015complexity,alon2016sign}. 

There are also antecedents to our study of the complex case. The signsolvability problem has been extended to the complex case in what is called the ray nonsingularity problem: when is every matrix with the same phases of a given square generalized phase matrix (in that literature named a \emph{ray pattern}) invertible? This question has been exploited and essentially solved for (nongeneralized) phase matrices in a series of papers at the turn of the millenium \cite{mcdonald1997ray,lee2000extremal,li2004non} and has seen a few subsequent developments. The notion of phase rank, introduced above, can be seen as a natural extension of this work. 

Lastly, in terms of historical precedents, the set of possible phases of points in an algebraic variety is what is called a \emph{coamoeba}. Coamoebas were introduced by Mikael Passare in talks in the early 2000s as a dual notion to that of amoebas. There has been a growing body of research on these objects, a small selection of influential works are \cite{nilsson2010discriminant,forsgaard2015order,MR3204268}. Checking if the phase rank of a given phase matrix is less or equal than a fixed $k$ is the same as checking if it is in the coamoeba of the variety of complex matrices of rank at most $k$, so it is a particular instance of the coamoeba membership problem. Again, this mimics the relation between the phaseless rank, introduced in \cite{GG21}, and the amoeba membership problem.

{\bf Organization:} In section \ref{sec:assocon} we introduce the definitions and results on sign rank and coamoebas that will be of use later in the paper. Following that, in Section \ref{sec:nonmaximal_phase} we survey the results on ray-nonsingularity through the lens of phase rank. In Section \ref{sec:smallsquare} we present the main result of the paper, a novel characterization of $3 \times 3$ matrices with phase rank $2$, that implies in particular that the coamoeba of the $3\times 3$ determinantal variety is completely determined by colopsidedness. Additionally we briefly discuss the open questions in the $4\times 4$ case. Finally in Section \ref{sec:bounds_phase} we observe that the lower bounds from the sign-rank can be adapted to our case and present a simple new upper bound. We end with a brief conclusion.

\section{Associated concepts}\label{sec:assocon}

In this section we set up the background of the related concepts of sign rank and coamoebas.

\subsection{Sign rank}\label{subsec:sign}

Sign matrices, with $\pm1$ entries, arise naturally in many areas of research. For instance, they are used to represent set systems and graphs in combinatorics, hypothesis classes in learning theory, and boolean functions in communication complexity. The minimum rank of a matrix with a given sign pattern has several important interpretations in these fields and has attracted a high level of interest in the last decade.

\begin{definition}
For a real matrix $M$ with no zero entries, let $\text{sign}(M)$ denote the sign matrix such that $(\text{sign}(M))_{ij} = \text{sign}(M_{ij})$, for all $i,j$. The sign rank of a sign matrix $S$ is defined as
$$\text{sign-rank}(S) = \min\{\rank(M) : \text{sign}(M) = S\}.$$
\end{definition}

Sign matrices are instances of phase matrices and it is straightforward to see that sign rank is the restriction of phase rank to sign matrices:
$$\rank_{\text{phase}}(S)=\text{sign-rank}(S), \, \text{for any sign matrix }S.$$
Sign rank computation is typically extended to matrices with entries in $\{-1,0,1\}$, which we call generalized sign matrices, where the zeros in the sign pattern force the corresponding entries to be zero.

Computing sign rank is hard. More precisely, in \cite{bhangale2015complexity} it is shown that proving that the sign rank of an $n \times m$ sign matrix is at most $k$ is NP-hard for $k \geq 3$. In particular, computing phase rank is also NP-hard, as the sign rank is the restriction of phase rank to a specific set of matrices. For $k=2$ they provide a polynomial time algorithm, which is not obviously generalizable to the complex case, suggesting a first question.

\begin{question} \label{q:rank2}
Is there a polynomial time algorithm to decide if an $n \times m$ phase matrix has phase rank at most $2$?
\end{question}

Checking if a generalized sign matrix has maximal rank is an interesting problem on its own, known also as the signsolvability problem. It is NP-hard \cite{klee1984signsolvability} and has a very simple characterization.

\begin{lemma}[Remark 1.1. \cite{klee1984signsolvability}] \label{lem:lopsign}
For an $n \times m$ generalized sign matrix $S$, with $m \geq n$, $\text{sign-rank}(S)=n$ if and only if every scaling of its rows by scalars in $\{-1,0,1\}$ has a unisigned nonzero column, i.e., a column that has only zeros and $1$'s or $-1$'s but not both.
\end{lemma}

In particular, a nongeneralized $n \times m$ sign matrix has maximal sign rank if and only if for any vector
$x \in \{-1,1\}^n$ either $x$ or $-x$ appears in its columns, which immediately implies that $n \times m$ sign matrices have nonmaximal sign rank for $m < 2^{n-1}$. Therefore, square $n \times n$ nongeneralized sign matrices cannot be sign-nonsingular for $n \geq 3$. In fact, Alon et al \cite{alon1985geometrical} show that the maximum sign rank of an $n\times n$ sign matrix is lower-bounded by $\frac{n}{32}$ and upper-bounded by $\frac{n}{2}(1+o(1))$.

Since computing sign rank is generally hard, a lot of effort was put into devising effective bounds for this quantity. For example, a well-known lower bound on the sign rank of an $m\times n$ sign matrix is due to Forster \cite{MR1964645}:

\begin{theorem}\label{thm:lbound_sign}
If $S$ is an $m\times n$  sign matrix, then $$\text{sign-rank}(S)\geq \frac{\sqrt{mn}}{||S||} ,$$ where $||S||$ is the spectral norm of $S$.
\end{theorem}

Alon et al. \cite{alon2016sign} observed that Forster's proof argument works as long as the entries of the matrix are not too close to zero, so we can replace $||S||$ with $||S||^*=\min\{||M||:M_{ij}S_{ij}\geq 1 \, \text{ for all }i,j\} = \min\{||M||:\text{sign}(M) = S \text{ and } |M_{ij}|\geq 1 \, \forall i,j\}$ in the above bound, which constitutes an improvement, since $||S||^*\leq ||S||$. Another improvement to Forster's bound, through a different approach, can be found in \cite{linial2007complexity} and it is based on the factorization of linear operators and uses the $\gamma^*_2$ norm the dual norm to the $\gamma_2$ norm.

\begin{theorem}[\cite{linial2007complexity}]\label{thm:lbound_sign2}
For every  $m\times n$ sign matrix $S$, $$\text{sign-rank}(S)\geq \frac{mn}{\gamma^*_2(S)}.$$
\end{theorem}

\begin{example}Let
$$S=\begin{bmatrix}
     1 & 1 & 1 & 1 & -1 & 1\\
    -1 & 1 & -1 & -1 & -1 & 1\\
    -1 & 1 & -1 & 1 & 1 & 1\\
    -1 & -1 & 1 & -1 & -1 & 1\\
     1 & 1 & -1 & -1 & -1 & -1\\
     1 & -1 & -1 & -1 & -1 & 1
\end{bmatrix}.$$
For this matrix we have $\frac{\sqrt{mn}}{||S||}=1.7990$  and $\frac{mn}{\gamma^*_2(S)}=2.0261.$
Thus, $$\text{sign-rank}(S)\geq 3.$$
\end{example}

Given a particular sign matrix $S$, there are not that many tools to upper bound its sign-rank, besides the ones that are given by any explicit matrix with the given sign pattern. A general bound, based only on the dimension of the matrix, can be obtained using the probabilistic method.

\begin{theorem} [Theorem 13.3.1 \cite{MR1885388}] \label{thm:ubound_sign}
For any $m\times n$ sign matrix $S$, $$\text{sign-rank}(S) \leq \frac{\min\{m,n\}+1}{2}+\sqrt{\frac{\min\{m,n\}-1}{2}\log(\max\{m,n\})}.$$
\end{theorem}

One can see that the above inequality is not tight. For $n=m=3$, for instance, the sign rank of any sign matrix is at most $2$, but the bound only guarantees that it is at most $3$. Is is believed, however, that asymptotically this bound is near optimal.

\subsection{Coamoebas of determinantal varieties}

As mentioned in the introduction, coamoebas are the images of varieties under taking the entrywise arguments, that is, under the map $$\text{Arg}: {(\mathbb{C}^{*})}^{n} \longrightarrow {(S^1)}^n, z=(z_1,\ldots,z_n)\rightarrow \bigg(\frac{z_1}{|z_1|},\ldots,\frac{z_n}{|z_n|}\bigg).$$
We will identify $S^1$ with the interval $[0,2\pi)$ in the usual way for purposes of plotting.

\begin{definition}
Given a complex variety $V \subseteq \mathbb{C}^{n}$, its \textit{coamoeba} is the set
$$\text{co}\mathcal{A}(V)=\{ \text{Arg}(z)=\bigg(\frac{z_1}{|z_1|},\ldots,\frac{z_n}{|z_n|}\bigg):z\in V \cap {(\mathbb{C}^{*})}^{n}\}.$$
\end{definition}

The connection between phase rank and coamoebas is clear. Given positive integers $n,m$ and $k$, with $k \leq \min\{n,m\}$, we define the determinantal variety $Y_{k}^{n,m}$ as the set
of all $n\times m$ complex matrices of rank at most $k$, cut out by the $k+1$ minors of an $n\times m$ matrix of variables. Directly from the definition of coamoeba, we have that the locus of $n\times m$ matrices of phase rank at most $k$ is the coamoeba of the determinantal variety $Y_{k}^{n,m}$. Computing phase rank is therefore a special case of the problem of checking coamoeba membership, which is a notoriously hard problem in general. There is however a necessary condition for coamoeba membership, called colopsidedness condition (\cite{forsgaard2015order},\cite{de2013geometry},\cite{forsgaard2012hypersurface}).

\begin{definition}
A sequence $z_1,z_2,\ldots,z_k \in \mathbb{C}$ is colopsided if, when considered as a set of points in $\mathbb{R}^2$, $0$ is not in the relative interior of its convex hull.

Given a finite multivariate complex polynomial $f(z) = \sum_{\alpha} b_{\alpha} z^{\alpha} \in \mathbb{C}[z_1,\ldots,z_n]$, $b_{\alpha}\neq 0$, and $\phi \in {(S^1)}^n$, we say that $f$ is colopsided at $\phi$ if the sequence
$$\bigg(\frac{b_{\alpha}}{|b_{\alpha}|}\phi^{\alpha}\bigg)_{\alpha}$$
is colopsided.
\end{definition}

For $S\subseteq \mathbb{R}^2$, we will denote its convex hull by $\text{Conv}(S)$. In practice, $0$ is in the relative interior of $\text{Conv}(S)$ if it can be written as a convex combination of the points in $S$ with all the coefficients strictly positive.

It is easy to see the following relation between colopsidedness and coamoebas.
\begin{lemma}\label{lem:colop_coamoeba}
If $\phi \in \text{co}\mathcal{A}(V(I))$, then no $f \in I$ is colopsided at $\phi$.
\end{lemma}

Note that the converse statement is, in general, not true. Not being colopsided for any $f \in I$ does not necessarily imply membership in the coamoeba $\text{co}\mathcal{A}(V(I))$.

\section{Nonmaximal phase rank}\label{sec:nonmaximal_phase}

While determining in general if the phase rank is less or equal than a certain $k$ is a hard problem, two cases are at least somewhat understood. The first is the case of rank one case, which reduces to the usual rank.

\begin{proposition} \label{prop:rank1phase}
For any phase matrix $\Theta$, $\rank_{\text{phase}}(\Theta)=1$ if and only if $\rank (\Theta)=1$.
\end{proposition}
\begin{proof}
Since $\rank_{\text{phase}}(\Theta) \leq \rank (\Theta)$, the result follows immediately if one can show that $\rank_{\text{phase}}(\Theta) =1$ implies that $\rank(\Theta)=1$. But if $M$ is the rank one matrix with the same phases as $\Theta$, we have $M=v w^t$ for some complex vectors $v$ and $w$. Let $\bar{v}$ and $\bar{w}$ be the vectors whose entries are the phases of $v$ and $w$, respectively. It is trivial to see that $\Theta = \bar{v} \bar{w}^t$, which completes the proof.
\end{proof}

The other is the case of nonmaximal rank. As mentioned earlier, the problem of determining under which conditions is the phase rank of a generalized phase matrix nonmaximal has been studied previously in the literature, under the name of ray nonsingularity \cite{mcdonald1997ray}, with a focus mostly on the square case. In the remainder of this section we will mainly present a survey of the implications of that work in the study of phase rank, with a few additional insights.

The first important observation is that we still have a version of Lemma \ref{lem:lopsign} for rectangular matrices.

\begin{lemma}[\cite{mcdonald1997ray}] \label{lem:lopphase}
An $n \times m$ phase matrix $\Theta$, with $m \geq n$, has $\rank_{\text{phase}}(\Theta)<n$ if and only if there is a scaling of its rows by scalars in $S^1 \cup \{0\}$, not all zero, such that no column is colopsided.
\end{lemma}

\begin{example} \label{ex:3x3phase}
Consider the matrix
$$ \Theta = \begin{bmatrix}
     1  & 1                    & 1 \\
     i  & e^{i \frac{\pi}{4}}  & e^{i 2\frac{\pi}{3}} \\
    -i  & e^{i \frac{7\pi}{6}} & e^{i 4\frac{\pi}{3}}
\end{bmatrix}.$$
The convex hulls of the entries in each column are shown in Figure \ref{fig:columnscolopsided}. Note that $0$ is in the convex hull of the entries in the first column but not in its (relative) interior, so this column is colopsided.
\begin{figure}[H]
  \centering
    \centerline{\includegraphics[width=0.3\textwidth]{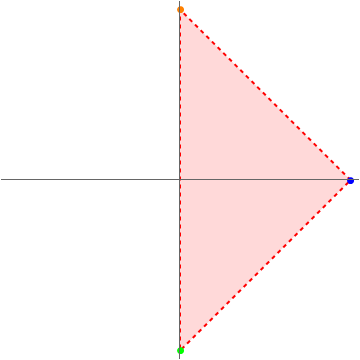} \hfill \includegraphics[width=0.3\textwidth]{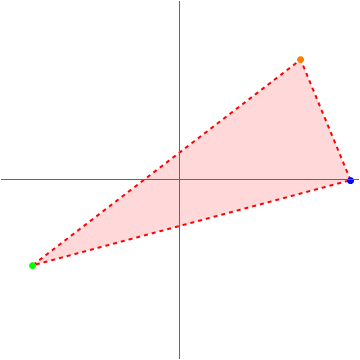} \hfill  \includegraphics[width=0.3\textwidth]{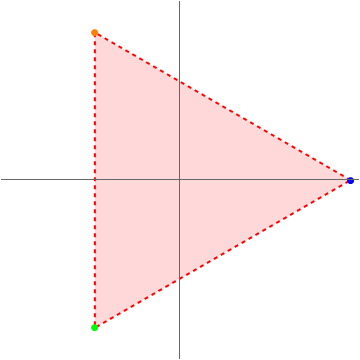}}
	\caption{Convex hulls of the entries of the columns of $\Theta$}
	\label{fig:columnscolopsided}
\end{figure}
Scaling a row of $\Theta$ by $e^{i\theta}$ ($\theta>0$) corresponds to rotating counterclockwise the points in the diagrams associated to that row through an angle of $\theta$. If we multiply the second row of $\Theta$ by $e^{i \frac{\pi}{4}}$, that is, if we rotate counterclockwise through an angle of $\frac{\pi}{4}$ the points in orange, we see that $0$ will be in the interior of all convex hulls (see Figure \ref{fig:columnscolopsided2}). 
\begin{figure}[H]
  \centering
    \centerline{\includegraphics[width=0.3\textwidth]{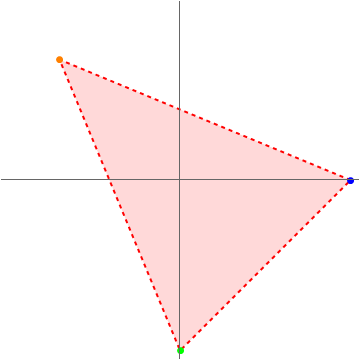} \hfill \includegraphics[width=0.3\textwidth]{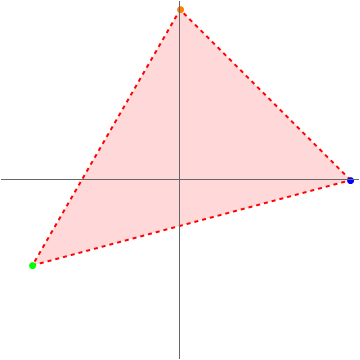} \hfill  \includegraphics[width=0.3\textwidth]{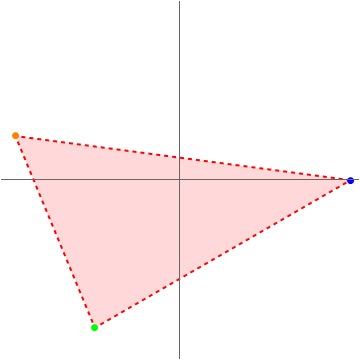}}
	\caption{Convex hulls of the entries of the columns of $\Theta$ after scaling the second row by $e^{i \frac{\pi}{4}}$.}
	\label{fig:columnscolopsided2}
\end{figure}
\end{example}

The criterion in Lemma \ref{lem:lopphase} is not necessarily easy to use, as searching for such a noncolopsided scaling is a nonconvex problem. There is another geometric way of thinking on this criterion. Consider an $n \times m$ phase matrix $\Theta$ and suppose that any submatrix obtained by erasing one of its rows has maximal phase rank (otherwise $\Theta$ is phase rank deficient). This is equivalent to asserting that in the row scalings of Lemma \ref{lem:lopphase} we only need to use scalars in $S^1$, as we will not zero out any row. Define $$\text{Colop}(n)=\{ y \in (S^1)^n \ : \ (y_1,...,y_n) \text{ is colopsided}\}.$$ The condition that all row scalings have a colopsided column is then easy to characterize geometrically using this set.

\begin{proposition}  \label{prop:lopgeom}
Let $\Theta$ be an $n \times m$ phase matrix, with $m \geq n$,  and assume that any submatrix obtained from $\Theta$ by erasing one of its rows has phase rank $n-1$. Then,
$\rank_{\text{phase}}(\Theta)=n$ if and only if
$$\cup_{j=1}^m (\text{Colop}(n)/\Theta^j) = (S^1)^n$$
where $\Theta^j$ is the $j$-th column of $\Theta$ and the division operation considered is the entrywise division.
\end{proposition}
\begin{proof}
Notice that
$$\text{Colop}(n)/\Theta^j = \{y \in (S^1)^n : (y_1 \Theta_{1j}, y_2 \Theta_{2j},...,y_n \Theta_{nj}) \text{ is colopsided}\}$$
is the set of all nonzero scalings that will make column $j$ colopsided. The union of all such sets is the collection of all scalings that will make some column colopsided, so if it does not cover the whole set of possible scalings, there will be a scaling that satisfies the conditions of Lemma \ref{lem:lopphase}.
\end{proof}

\begin{example}
When computing the phase rank of a phase matrix one may always assume the matrix first row to be all $1$'s. Moreover, since multiplying all scalars by a common scalar in $S^1$ does not change the colopsidedness of columns, we may also assume that the scalars vector in Proposition \ref{prop:lopgeom} has the first entry equal to one. For $n=3$, we can then think of the elements of $\text{Colop}(3)$ as pairs of angles (see Figure \ref{fig:colopset}). Although we usually identify $S^1$ with $[0,2\pi)$, here we plotted several fundamental domains to better visualize the pattern and the toric nature of this object. Note further that the vertices of the hexagons are actually not in $\text{Colop}(3)$.
\begin{figure}[H]
  \centering
    \centerline{\includegraphics[width=0.3\textwidth]{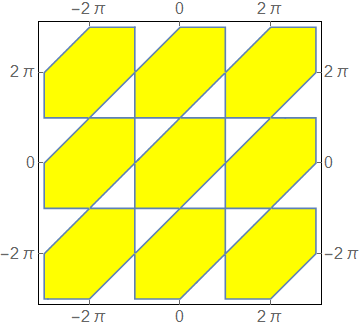} }
	\caption{Representation of $\text{Colop}(3)$.}
	\label{fig:colopset}
\end{figure}

Given a $3 \times m$ matrix $\Theta$, we can easily check if there is any $2 \times m$ submatrix of nonmaximal phase rank, since that is equivalent to having usual rank $1$, by Proposition \ref{prop:rank1phase}. So we can see if we are in the conditions of Proposition \ref{prop:lopgeom}, and checking if $\rank_{\text{phase}}(\Theta)=n$ is the same as checking if the geometric translations of the region in Figure \ref{fig:colopset} by the symmetric of the angles of the columns of $\Theta$ cover the entire space. Let $\Theta$ be the matrix from Example \ref{ex:3x3phase}. We have to consider the translations $\text{Colop}(3) - (\pi/2,-\pi/2)$, $\text{Colop}(3) - (\pi/4, 7 \pi/6)$ and $\text{Colop}(3) - (2\pi/3,4\pi/3)$. We get the regions shown in the left of Figure \ref{fig:colopcheck}, which do not cover the plane, so $\rank_{\text{phase}}(\Theta)<3$.
\begin{figure}[H]
  \centering
    \centerline{\includegraphics[width=0.3\textwidth]{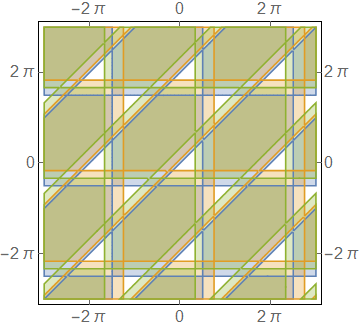} \hspace{2cm} \includegraphics[width=0.3\textwidth]{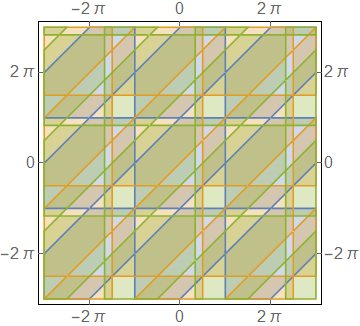}}
	\caption{Representation of the translations of $\text{Colop}(3)$ for two different matrices.}
	\label{fig:colopcheck}
\end{figure}
It is not hard to find examples where the translations do cover the entire space. For instance, for
$$ \Theta = \begin{bmatrix}
     1  & 1                    & 1 \\
     1  & i  & e^{i 2\frac{\pi}{3}} \\
    1   & -i & e^{i \frac{\pi}{6}}
\end{bmatrix},$$
we obtain the region shown in the right of Figure \ref{fig:colopcheck}, that indeed covers the whole space, so this matrix has phase rank $3$.
\end{example}

Note that, although in the previous example we saw that $\text{Colop}(3)$ occupies a large portion of $(S^1)^3$, this ceases to be true as $n$ increases. In fact, if we think of $\text{Colop}(n)$ as a subset of $[0,2 \pi )^n$, it is not hard to compute its volume.

\begin{lemma}
For $n \geq 2$, the volume of $\text{Colop}(n)$ seen as a subset of $[0,2\pi)^n$ is $2 n \pi^n$.
\end{lemma}
\begin{proof}
A set $\Gamma$ of $n$ angles (defined up to additions of multiples of $2\pi$) is colopsided (seen as a subset of $(S^1)^n$) if they are all in the same halfspace, i.e., if there is some $\alpha \in \Gamma$ such that every other angle in $\Gamma$ is in $[\alpha,\alpha + \pi]$. There are $n$ choices for $\alpha$ so we get the volume
$$n \int_{0}^{2\pi} \int_{[\alpha,\alpha+\pi]^{n-1}} dV d\alpha = 2 n \pi^n.$$
\end{proof}

This immediately leads to a sufficient condition for nonsingularity, based on this lemma and Lemma \ref{lem:lopphase}.

\begin{proposition}[\cite{lee2000extremal}] \label{prop:nonsingularphaserec}
If $n \leq m < \frac{2^{n-1}}{n}$, then $\rank_{\text{phase}}(\Theta)<n$ for every $n \times m$ phase matrix $\Theta$.
\end{proposition}
\begin{proof} Let $\Theta$ be an $n \times m$ phase matrix. Just note that the volume of $\cup_{j=1}^m (\text{Colop}(n)/\Theta^j)$ is at most $m$ times the volume of $\text{Colop}(n)$. If that is not enough to cover $(S^1)^n$, which has volume $(2\pi)^n$, that means there is a set of scalars that do not make any column of $\Theta$ colopsided. Therefore, by Lemma \ref{lem:lopphase}, the matrix is phase rank deficient. This happens if
$$m 2 n \pi^n < (2\pi)^n,$$
which gives us the intended result.
\end{proof}

This result was proven in \cite{lee2000extremal} using a similar but distinct argument that used a combinatorial gadget instead of this one based on volumes. A nice thing about the argument presented here is that it makes clear that the bound is somewhat conservative: in order to have a $n \times \frac{2 ^{n-1}}{n}$ phase matrix with full rank, we would need $\frac{2 ^{n-1}}{n}$ translated copies of $\text{Colop}(n)$ that covered the entire space without overlapping in any positive volume set. This is a very restrictive condition. In general, one would expect that much more copies than this will be needed to actually cover the space. This is, however, enough to take care of most square cases.

\begin{corollary}
For $n \geq 7$, no $n \times n$ phase matrix has phase rank $n$.
\end{corollary}
\begin{proof}
Just note that for $n \geq 7$, $\frac{2^{n-1}}{n} > n$.
\end{proof}

In \cite{lee2000extremal} the authors were able to slightly modify this argument to prove that there are also no $6 \times 6$ phase matrices with phase rank $6$, while in \cite{li2004non} a very long and technical proof was provided to show that there are also no $5 \times 5$ phase rank nonsingular matrices. For $n \leq 4$ such matrices exist and we will come back to these in the next section.

The rectangular case is still not fully characterized. We know from Lemma \ref{lem:lopsign} on sign rank that the $n \times 2^{n-1}$ matrix whose columns are all the distinct $\pm 1$ vectors starting with a one has sign rank (and thus phase rank) $n$, and we just saw that for $m < \frac{2^{n-1}}{n}$ these do not exist.

\begin{question}
What is the smallest $m$ for which there is a $n \times m$ phase matrix with phase rank $n$?
Can one at least find better bounds than $\frac{2^{n-1}}{n} \leq m \leq 2^{n-1}$?
\end{question}

As explained before, this can be seen as the problem of the minimal number of geometric translations of a certain set that are needed to cover $(S^1)^n$, a type of geometrical problem that tends to be hard.

\section{Small square matrices of maximal phase rank}\label{sec:smallsquare}

We have seen that for $n \geq 5$ every $n \times n$ phase matrix has phase rank less than $n$. So the question of determining if a square matrix is phase nonsingular is only potentially interesting for $n=2,3,4$.

For $n=2$, due to Proposition \ref{prop:rank1phase}, the phase rank coincides with the usual rank, so the question is trivial. For $n=3$ and $n=4$, Lemma \ref{lem:lopphase} gives a potential way to find if the rank is maximal, but it involves either finding a scaling with the required properties or proving that one does not exist, and there is no direct way of doing that. So one would like a simple, preferably semialgebraic, way of describing the sets of $n \times n$ phase  matrices of  phase rank less than $n$, for $n=3$ and $4$.

A sufficient certificate for maximal phase rank is given by the colopsided criterion in coamoeba theory, as stated in Lemma \ref{lem:colop_coamoeba}. This method is already proposed in \cite{mcdonald1997ray}.

\begin{definition}
We say that an $n\times n$ phase matrix $\Theta$ is colopsided if the $n \times n$ determinant polynomial is colopsided at $\Theta$.
\end{definition}

By Lemma \ref{lem:colop_coamoeba}, if $\Theta$ is colopsided, $\Theta$ does not belong to the coamoeba of the determinant, i.e., it has phase rank $n$. Let us denote by $\overrightarrow{\det(\Theta)}$ the vector of the monomials of the $n \times n$ determinant evaluated at the phase matrix $\Theta$. We are saying that if $0$ is not in the relative interior of the convex hull of $\overrightarrow{\det(\Theta)}$ then $\Theta$ has phase rank $n$.

\begin{example}\label{ex:colops}
Let $$\Theta_1=\begin{bmatrix}
1 & 1 & 1\\
1 & e^{i\frac{3\pi}{4}}  &   e^{-i\frac{\pi}{2}}\\
1 & e^{-i\frac{\pi}{2}} &  e^{i\frac{3\pi}{4}}
\end{bmatrix}, \, \Theta_2=\begin{bmatrix}
1 & 1 & 1\\
1 & e^{i\frac{\pi}{2}}  &   e^{i\frac{\pi}{2}}\\
1 & e^{i\frac{\pi}{2}} &  -1
\end{bmatrix} \text{ and }\Theta_3=\begin{bmatrix}
1 & 1 & 1\\
1 & e^{i\frac{\pi}{2}}  &   e^{i\frac{\pi}{3}}\\
1 & e^{i\frac{\pi}{3}} &  e^{i\frac{\pi}{2}}
\end{bmatrix}.$$
We have 
\begin{align*}
\overrightarrow{\det(\Theta_1)} & =  (e^{i\frac{3\pi}{2}}, e^{-i\frac{\pi}{2}},e^{-i\frac{\pi}{2}},-e^{i\frac{3\pi}{4}},-e^{i\frac{3\pi}{4}},1),\\ 
\overrightarrow{\det(\Theta_2)} & =  (-e^{i\frac{\pi}{2}}, e^{i\frac{\pi}{2}},e^{i\frac{\pi}{2}},-e^{i\frac{\pi}{2}},1,1),\\
\overrightarrow{\det(\Theta_3)} & =  (-1, e^{i\frac{\pi}{3}}, e^{i\frac{\pi}{3}}, -e^{i\frac{\pi}{2}}, -e^{i\frac{\pi}{2}},-e^{i\frac{2\pi}{3}},-e^{i\frac{\pi}{2}}).
\end{align*}  
Both $\Theta_1$ and $\Theta_2$ are colopsided, while $\Theta_3$ is not (see Figure \ref{fig:colops}). $\text{Conv}(\overrightarrow{\det(\Theta_1)})$ does not contain $0$. $\text{Conv}(\overrightarrow{\det(\Theta_2)})$ does contain it, but its relative interior does not. We can immediately conclude that both $\Theta_1$ and $\Theta_2$ have phase rank $3$.

\begin{figure}[H]
  \centering
    \centerline{\includegraphics[width=0.24\textwidth]{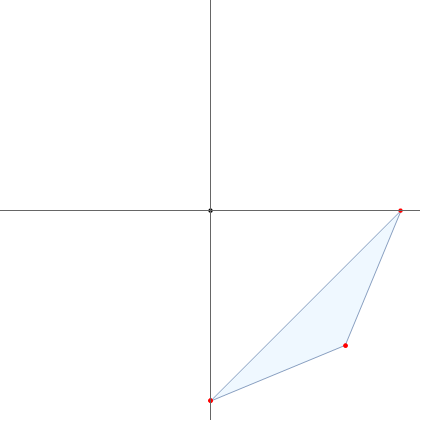} \hspace{2cm} \includegraphics[width=0.24\textwidth]{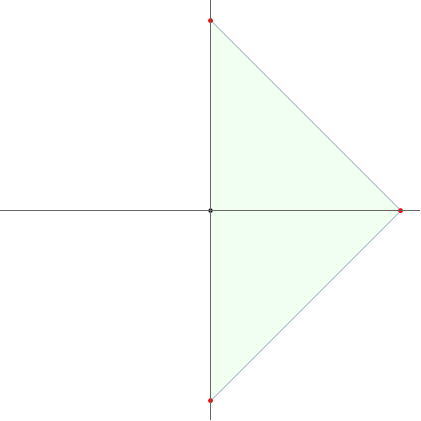} \hspace{2cm} \includegraphics[width=0.24\textwidth]{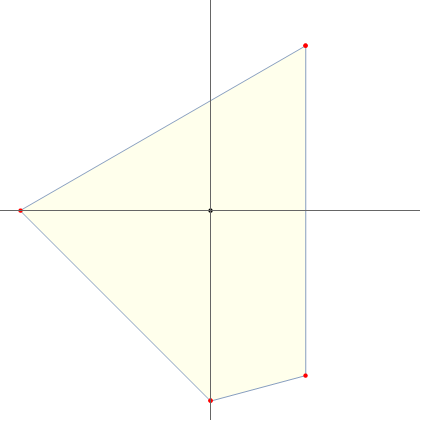}}
	\caption{Convex hulls of $\protect\overrightarrow{\det(\Theta_1)}, \protect\overrightarrow{\det(\Theta_2)}, \protect\overrightarrow{\det(\Theta_3)}$, in this order, where $\Theta_1$, $\Theta_2$ and $\Theta_3$ are the matrices from Example \ref{ex:colops}.}
	\label{fig:colops}
\end{figure}
\end{example}

The colopsidedness criterion was proposed in the ray nonsingularity literature, and was noted not to be necessary in the $4 \times 4$ case. In fact, for
$$\Theta=\begin{bmatrix}
1 & 1 & 1 & i  \\
1 & 1 &-1 & 1  \\
1 & -1 & 1 & 1 \\
-1 & 1 & 1 & 1 \\
\end{bmatrix},$$
$\overrightarrow{\det(\Theta)}$ contains $1,-1,i$ and $-i$,
so $0$ is in the interior of its convex hull. Moreover, in \cite{lee2000extremal} it is shown that this matrix has phase rank $4$.

The question of necessity of the colopsided criterion for phase nonsingularity is not addressed in the ray nonsingularity literature. In \cite{li2004non}, for example, in order to study the $5 \times 5$ case, the authors derive an extensive and complicated description for the $3\times3$ case, without any reference to colopsidedness. It turns out that colopsidedness is a necessary and sufficient condition for phase nonsingularity of $3 \times 3$ matrices.

\begin{theorem}\label{thm:phasecolop3x3}
Given a $3 \times 3$ phase matrix $\Theta$, $\rank_{\text{phase}}{\Theta} < 3$ if and only if the origin is in the relative interior of the convex hull of  $\overrightarrow{\det(\Theta)}$.
\end{theorem}

In terms of coamoeba theory, we are saying that the coamoeba of the variety of singular $3\times 3$ matrices is characterized by the noncolopsidedness of the determinant. Coamoebas of simple polynomials are completely characterized by the colopsidedness criterion (see \cite{forsgaard2012hypersurface}). This is a new nontrivial example of another hypersurface with the same property, since the $3 \times 3$ determinant is not simple.

In what follows we will present a proof of Theorem \ref{thm:phasecolop3x3}, based only on simple results from linear algebra and convex geometry. We will consider throughout a general $3 \times 3$ phase matrix of the following form:
$$\Theta=\begin{bmatrix}
e^{i \phi_1} & e^{i \phi_2} & e^{i \phi_3}\\
e^{i \phi_4} & e^{i \phi_5} & e^{i \phi_6} \\
e^{i \phi_7} & e^{i \phi_8}  & e^{i \phi_9}
\end{bmatrix}.$$
For this matrix, $\overrightarrow{\det(\Theta)}$ equals \begin{small}
$$\big(e^{i(\phi_1 + \phi_5 + \phi_9)}, e^{i(\phi_2 + \phi_6 + \phi_7)}, e^{i(\phi_3 + \phi_4 + \phi_8)}, -e^{i(\phi_1 + \phi_6 + \phi_8)}, -e^{i(\phi_2 + \phi_4 + \phi_9)}, -e^{i(\phi_3 + \phi_5 + \phi_7)}\big).$$
\end{small}
We will prove the equivalence between noncolopsidedness and nonmaximal phase rank for the $3\times 3$ case in two steps. First, we show that nonmaximal phase rank is the same as noncolopsidedness with one additional condition and, then, that this extra restriction can be removed.

\begin{lemma}\label{lem:aux_colop}
Let $\Theta$ be a $3\times 3$ phase matrix. Then, $\rank_{\text{phase}}(\Theta)<3$ if and only if there exists a coefficient vector $c \in \mathbb{R}^6_{++}$ such that $\overrightarrow{\det(\Theta)}\cdot c =0$ and  $c_1 c_2 c_3 = c_4 c_5 c_6.$
\end{lemma}

\begin{proof}
Let $\Theta$ be as above and suppose $\rank_{\text{phase}}(\Theta) < 3$. One can find a real positive matrix  $M$ such that $\det(M*\Theta)=0$, where $*$ represents the Hadamard product.
But $\det(M*\Theta)$ can be written as the dot product 
$\overrightarrow{\det(\Theta)}\cdot A$
where $A$ is the vector \begin{small}$$(M_{11}M_{22}M_{33},M_{12}M_{23}M_{31},M_{13}M_{21}M_{32},M_{11}M_{23}M_{32},M_{12}M_{21}M_{33},M_{13}M_{22}M_{31})$$\end{small}
whose entries satisfy the intended relations.

Conversely, suppose there exist positive coefficients $c_i,i=1,\ldots,6,$ satisfying $\overrightarrow{\det(\Theta)}\cdot c =0$ and  $c_1 c_2 c_3 = c_4 c_5 c_6.$ According to the reasoning above, if we can find a positive matrix $M$ for which $\det(M*\Theta)=\overrightarrow{\det(\Theta)}\cdot c =0$,  $\Theta$ will have nonmaximal phase rank. Finding this $M$ is equivalent to write $c$ as a vector with the form of $A$ above, which is equivalent to solving the linear system
\begin{align*}
M'_{11} + M'_{22} + M'_{23} = c'_1, \, M'_{12} + M'_{23} + M'_{31} = c'_2, \, M'_{13} + M'_{21} + M'_{32} = c'_3 \\
M'_{11} + M'_{23} + M'_{32} = c'_4, \, M'_{12} + M'_{21} + M'_{33} = c'_5, \, M'_{13} + M'_{22} + M'_{31} = c_6',
\end{align*} where $c'_i = \log c_i$ and  $M'_{ij} = \log M_{ij}.$
This is solvable for $M'_{ij}$ if and only if $c'$ is in $\mathcal{C}(B)$, the column space of
$$B=\begin{bmatrix}
1 & 0 & 0 & 0 & 1 & 0 & 0 & 0 & 1\\
0 & 1 & 0 & 0 & 0 & 1 & 1 & 0 & 0\\
0 & 0 & 1 & 1 & 0 & 0 & 0 & 1 & 0\\
1 & 0 & 0 & 0 & 0 & 1 & 0 & 1 & 0\\
0 & 1 & 0 & 1 & 0 & 0 & 0 & 0 & 1\\
0 & 0 & 1 & 0 & 1 & 0 & 1 & 0 & 0
\end{bmatrix},$$
or, equivalently, $c'$ is orthogonal to $\text{Null}(B^T)$, the null space of the transpose of $B$. Since $\text{Null}(B^T)$ is spanned by $(
1 1  1  -1  -1  -1)^\intercal$, $c'$ is in $\mathcal{C}(B)$ if and only if $c'_1 + c'_2 + c'_3 = c'_4 + c'_5 + c'_6$, i.e., $c_1 c_2 c_3 = c_4 c_5 c_6$, which is guaranteed by our hypothesis.
\end{proof}

To get rid of this extra condition on the coefficients we will have to do some extra work. The following lemma is a simple fact in discrete convex geometry and it will be useful later.

\begin{lemma}\label{lem:redbluepoints}
Given $3$ blue points and $3$ red points in $\RR^2 \setminus \{(0,0)\}$ such that the origin is in the relative interior of their convex hull, there is a proper subset of them containing all the blue points that still has the origin in the relative interior of its convex hull (and similarly for the red points).
\end{lemma}

\begin{proof}
It is sufficient to prove that we can always erase some red points while keeping the origin in the relative interior, as the statement is clearly symmetric with respect to the colors.

Suppose that there is no triangle with vertices on the six given points containing the origin in its interior. Then there are two possibilities: either all the points are in a line through the origin, or all the points are in exactly two lines through the origin.
This is a consequence of a theorem of Steinitz \cite{BedingtkonvergenteReihenundkonvexeSysteme} (see Result B in \cite{bonnice1966interior}).

If all are in one line, then for each color there must be two red points in the same side of the origin, so we can drop one without losing the origin in the relative interior of the convex hull. If all are in exactly two lines then we must have points in each side of the origin on each line, and if we have more than one point on a side then we can remove one of them. If no red point is redundant, that means that all the blue points are in the same side of the origin in one of the lines, which means we can drop the line with only red points and we will still have the origin on the relative interior.

So we can assume that there is a triangle containing the origin in its interior. If there is such a triangle using a blue point, then we are done, as the set attained by adding any missing blue point to the vertices of the triangle would have the required properties.

The only remaining case would be if there is a triangle that uses only red points. In that case, if we take any of the blue points and consider the three triangles it can define with the red points, the only way that none of them contains the origin in its interior is if the origin is in the interior of the segment between the blue and a red point. If all the blue points are opposite to different red points, then the triangle of blue points would contain the origin, if not, then the set of blue points together with their opposite red points satisfies the intended property.
\end{proof}

We are now ready to prove Theorem \ref{thm:phasecolop3x3} by using this convex geometry fact to remove the extraneous condition on Lemma \ref{lem:aux_colop}.

\begin{proof}[Proof of Theorem \ref{thm:phasecolop3x3}]
If the origin is not in the relative interior of the convex hull of $\overrightarrow{\det(\Theta)}$, this means the determinant is colopsided at $\Theta$, which implies, by Lemma \ref{lem:colop_coamoeba}, that $\rank_{\text{phase}}(\Theta)=3$. So one of the implications is easy.

We will show that if $\Theta$ is not colopsided then there exists a coefficient vector $c \in \mathbb{R}^6_{++}$ such that $\overrightarrow{\det(\Theta)}\cdot c =0$ and  $c_1 c_2 c_3 = c_4 c_5 c_6.$ By Lemma \ref{lem:aux_colop}, this will imply that $\Theta$ has nonmaximal phase rank, giving us the remaining implication.	

Suppose $\Theta$ is not colopsided. Then, there exists $a \in \mathbb{R}^6_{++}$ such that
$\overrightarrow{\det(\Theta)}\cdot a =0.$
If $a_1 a_2 a_3 = a_4 a_5 a_6$, we are done. Thus, we either have $a_1 a_2 a_3 > a_4 a_5 a_6$ or $a_1 a_2 a_3 < a_4 a_5 a_6$. Without loss of generality we may assume the first, since switching two rows in the matrix will switch the sets $\{a_1,a_2,a_3\}$ and $\{a_4,a_5,a_6\}$.

Note that the entries of $\overrightarrow{\det(\Theta)}$ can be thought of as points in $\RR^2$ where the first three are red and the last three are blue, and their convex hull contains the origin in its relative interior. So, by Lemma \ref{lem:redbluepoints}, the origin is still in the relative interior of the convex hull if we drop some of the $a_i$'s, $i=1,2,3.$
This means that there exists $b \in \mathbb{R}^6_{+}$ for which some of the first three coordinates are zero but none of the last three coordinates is, and such that  $\overrightarrow{\det(\Theta)}\cdot b =0$.

Now note that every $6$-uple of the form $c^{\lambda}=a+\lambda b$, $\lambda \geq 0$, will satisfy $c^{\lambda} \in \mathbb{R}^6_{++}$ and $\overrightarrow{\det(\Theta)}\cdot c^{\lambda} =0$. Furthermore,
$$c^{\lambda}_1c^{\lambda}_2c^{\lambda}_3-c^{\lambda}_4c^{\lambda}_5c^{\lambda}_6=(a_1 + b_1 \lambda) (a_2 + b_2 \lambda) (a_3 + b_3 \lambda) - (a_4 + b_4 \lambda) (a_5 + b_5 \lambda) (a_6 + b_6 \lambda)$$
equals $a_1 a_2 a_3 - a_4 a_5 a_6 >0 $ for $\lambda = 0$ but goes to $-\infty$ when $\lambda$ grows to $+\infty$, since it is a cubic polynomial on $\lambda$ with the coefficient of $\lambda^3$ being $-b_4b_5b_6$. Hence, for some $\lambda$ we have that $c^{\lambda}_1c^{\lambda}_2c^{\lambda}_3-c^{\lambda}_4c^{\lambda}_5c^{\lambda}_6=0$ and there is a vector $c$ with the desired properties.
\end{proof}

\begin{example}
We saw in Example \ref{ex:colops} that the matrix $$\Theta_3=\begin{bmatrix}
1 & 1 & 1\\
1 & e^{i\frac{\pi}{2}}  &   e^{i\frac{\pi}{3}}\\
1 & e^{i\frac{\pi}{3}} &  e^{i\frac{\pi}{2}}
\end{bmatrix}$$
has a noncolopsided determinant. Therefore, $\rank_{\text{phase}}(\Theta_3)<3$.
\end{example}

\begin{example}
The characterization of nonsingular $3 \times 3$ phase matrices given by Theorem \ref{thm:phasecolop3x3} is very easy to check. In particular, it can be used to visualize  slices of the coamoeba associated to the $3\times3$ determinant, i.e., of the set of phase singular $3 \times 3$ phase matrices. In Figure \ref{fig:colop_nonmax} are shown, on the left, the set of triples $(t_1,t_2,t_3)$ for which $$\begin{bmatrix}
	1 & 1 & 1\\
	1 & e^{it_1} & e^{it_2}\\
	1 & e^{it_3} & e^{i\frac{\pi}{3}}
	\protect\end{bmatrix}$$  has nonmaximal phase rank, and, on the right, its complement.
\begin{figure}[h]
  \centering
  \centerline{\includegraphics[width=0.4\textwidth]{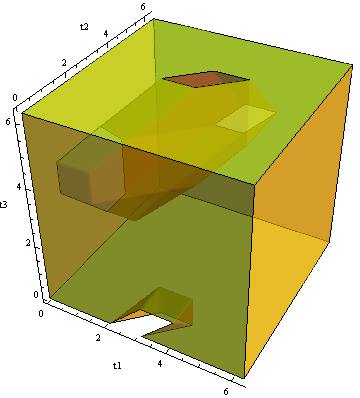} \hspace{2cm} \includegraphics[width=0.4\textwidth]{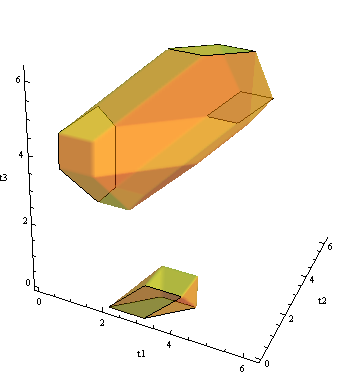}}
    \caption{Slice of the $3\times3$ determinant coamoeba and its complement.}
	\label{fig:colop_nonmax}
\end{figure}
\end{example}

Since checking if the phase rank of a $3 \times 3$ matrix is less or equal than two is so easy, it could be tempting to think that we might be able to leverage the result into some answer to Question \ref{q:rank2}, that asks for an algorithm to check if an $n \times m$ matrix has phase rank at most $2$. The problem is that, contrary to what happened in the phaseless rank case, not even the $3 \times m$ case can be easily derived from the rank of its $3 \times 3$ submatrices. In fact, the matrix
$$\begin{bmatrix}
1 & 1 & 1 & 1\\
1 & 1 & -1 & -1\\
1 & -1 & 1 & -1
\end{bmatrix}$$
has phase rank $3$, although all its $3 \times 3$ submatrices have sign rank (and thus phase rank) at most $2$.

As mentioned before, for $4 \times 4$ matrices we have seen that colopsidedness is not necessary for having phase rank $4$. In fact, even generating colopsided examples tends to be hard, but there are some systematic ways of doing it.

\begin{example}\label{ex:colop1}
The phase matrix $$\Theta=\begin{bmatrix}
1 & 1 & 1 & 1\\
1 & -1 & e^{i\frac{\pi}{4}} & e^{i\frac{\pi}{4}}\\
1 & e^{i\frac{\pi}{4}} & -1 & e^{i\frac{\pi}{4}}\\
1 & e^{i\frac{\pi}{4}} & e^{i\frac{\pi}{4}} & -1\\
\end{bmatrix}$$ is colopsided.
$$\text{Conv}(\overrightarrow{\det(\Theta)})=\text{Conv}(\{e^{i\frac{\pi}{2}},e^{i\frac{3\pi}{4}},-1,e^{i\frac{5\pi}{4}},e^{i\frac{3\pi}{2}}\})$$
contains $0$, but its (relative) interior does not (see Figure \ref{fig:colop1}).
\end{example}

\begin{figure}[H]
  \centering
    \centerline{\includegraphics[width=0.4\textwidth]{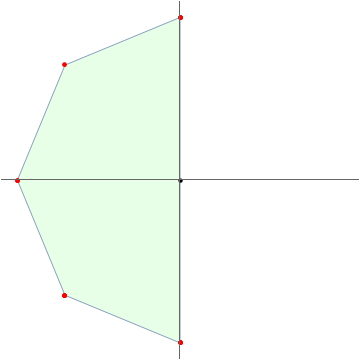}}
    \caption{Convex hull of $\protect\overrightarrow{\det(\Theta)}$, where $\Theta$ is the matrix from Example \ref{ex:colop1}.}
%     \caption{$\text{Conv}(\overrightarrow{\det(\Theta)})$, with
% 			$\Theta=\bigg[\protect\begin{smallmatrix}
% 			1 & 1 & 1\\
% 			1 & i & -i\\
% 			1 & 1 & i \protect\end{smallmatrix}\bigg].$}
	\label{fig:colop1}
\end{figure}

The underlying idea behind the generation of Example \ref{ex:colop1} was the following: we searched for a $3\times4$ phase matrix for which all $3\times 3$ submatrices are colopsided, namely $$\Theta=\begin{bmatrix}
1 & 1 & 1 & 1\\
1 & -1 & e^{i\frac{\pi}{4}} & e^{i\frac{\pi}{4}}\\
1 & e^{i\frac{\pi}{4}} & -1 & e^{i\frac{\pi}{4}}
\end{bmatrix}.$$
This means that for every $3 \times 3$ submatrix $\Theta_i$,
$\overrightarrow{\det(\Theta_i)}$ is colopsided, i.e., every point is contained in some
 common half-space $H_i$. Then, having in mind the Laplace expansion formula for $4\times 4$
 determinant along its last row, we can see that $\overrightarrow{\det(\Theta)}$ is just
 the collection of the sets $\overrightarrow{\det(\Theta_i)}$ each rotated by its
 complementary entry in the $4 \times 4$ matrix. By picking suitable entries we can
 rotate them in such a way that all the $H_i$ will coincide and ensure the whole matrix
 is colopsided. This gives us a tool to construct colopsided $4 \times 4$ matrices from
 $3 \times 4$ matrices where all $3\times 3$ submatrices have phase rank $3$. The problem is that, on the one hand,
 colopsidedness is not enough to fully characterize $4 \times 4$ phase nonsingularity, and, on the other hand, finding $3 \times 4$ matrices with this property is also hard. Thus, the problem of finding such a characterization remains completely open.

\begin{question}
Find an \emph{effective} description of the set of $4 \times 4$ phase matrices with phase rank less than $4$.
\end{question}

\section{Bounds for phase rank}\label{sec:bounds_phase}

\subsection{Lower bounds}

The lower bounds presented in Theorems \ref{thm:lbound_sign} and \ref{thm:lbound_sign2} for the signless rank can be directly transposed to the complex case. One only need to note that the original proofs of those results still go through, with very minor modifications to adapt them to the complex case. 

\begin{theorem}\label{thm:lwbd_phase}
If $\Theta$ is an $n\times m$ phase matrix, then $\rank_{\text{phase}}(\Theta) \geq \frac{\sqrt{nm}}{||\Theta||},$ where $||\Theta||$ is the spectral norm of $\Theta$,
\end{theorem}

Analogously to Theorem \ref{thm:lbound_sign}, we can replace $||\Theta||$ with  $$||\Theta||^*=\min\left\{||M||:\frac{M_{ij}}{|M_{ij}|}= \Theta_{ij} \text{ and } |M_{ij}|\geq 1 \, \forall i,j\right\}$$ in the above bound.

Theorem \ref{thm:lbound_sign2} can also be generalized to the complex case. This bound involves the dual norm of the $\gamma_2$ norm, also known as the max-norm. Recall that for a matrix $A \in \mathbb{C}^{n\times m}$ we have
$$\gamma_2(A)= \min\bigg\{\max_{i,j}||x_i||_{l_2} ||y_j||_{l_2}:XY^T=A\bigg\},$$
where $\{x_i\}_{i=1}^n$ and $\{y_j\}_{j=1}^m$ are the rows of $X$ and $Y$, respectively. Therefore, we have
$$\gamma^*_2(A)=\max_{B:\gamma_2(B)\leq 1} |\langle A,B \rangle|= \max_{B:\gamma_2(B)\leq 1} |\Tr(A^*B)|.$$

\begin{theorem} \label{thm:complexgamma2bound}
For any $n\times m$ phase matrix $\Theta$, $$\rank_{\text{phase}}(\Theta)\geq \frac{nm}{\gamma^*_2(\Theta)}.$$
\end{theorem}

Note that these bounds, being continuous, can only be very rough approximations to the discontinuous notion of phase rank, but can in some cases provide meaningful results and have been used in the real case to prove several results in communication complexity.

\begin{example}

Searching randomly for small examples one immediately observes the limitations of these bounds, as it is very hard to obtain effective bounds.

For instance, the bounds from Theorems \ref{thm:lwbd_phase} and \ref{thm:complexgamma2bound} for the matrix
$$\begin{bmatrix}
1 & 1 & 1 & i  \\
1 & 1 &-1 & 1  \\
1 & -1 & 1 & 1 \\
-1 & 1 & 1 & 1 \\
\end{bmatrix},$$
which has phase rank $4$, are $1.5751$  and $1.6996$, respectively. The smallest example that we could find for which one of the bounds is nontrivial has size $8 \times 8$ and is shown below. For
$$\begin{bmatrix}
    -1 &    i  &  -1  &   i  &   i &   -1  &   1  &   1 \\
    -1 &    i  &   i  &   i  &  -i &   -i &   -i  &   i\\
     1 &    i  &   1  &  -1  &   i &    i &   -i  &  -1 \\
     1 &    i  &   1  &   1  &   i &    1 &    1  &   i \\
     1 &    i   & -1  &   i  &   i &   -i &   -1  &   1 \\
    -i &    -i &   1  &  -i  &   i &   -1 &   -i  &  -1 \\
    -1 &    i  &  -1  &  -i  &   i &    i &   -1  &  -i \\
     i &   -1  &   i  &  -i  &   1 &   -1 &   -i  &  -1
\end{bmatrix}$$
the bounds values are $1.8383$ and $2.0335$, so Theorem \ref{thm:complexgamma2bound} guarantees that it has phase rank at least $3$. Note that one could maybe conclude the same from checking that some of the $3 \times 3$ submatrices are colopsided, but to attain lower bounds larger than $3$ these inequalities are our only systematic tool.
\end{example}

\subsection{Upper bounds}

One obvious upper bound for the phase rank is the usual rank, but it has a very limited scope. In terms of upper bounds for phase rank that depend only on the size of the matrix, one could aim for a generalization of Theorem \ref{thm:ubound_sign}. The issue lies in the fact that the proof uses the combinatorial nature of a sign matrix, such as the number of sign changes on each row of the matrix, and it is not obvious how this concept applies to a phase matrix. We can however use Proposition \ref{prop:nonsingularphaserec} to create a very simple upper bound.

\begin{proposition} \label{prop:ub_phase}
Let $\Theta$ be an $n \times m$ phase matrix, with $n \leq m$, and $k$ the smallest positive integer such that
$m < \frac{2^{k-1}}{k}$. Then $\rank_{\text{phase}}(\Theta) \leq n-\lfloor \frac{n-1}{k-1} \rfloor$.
\end{proposition}

\begin{proof}
Let $k$ satisfy $m < \frac{2^{k-1}}{k}$. Then, by  Proposition \ref{prop:nonsingularphaserec}, every $k\times m$ submatrix of $\Theta$ has phase rank at most $k-1$. Hence, for every $k\times m$ submatrix $\Gamma$, we can find $B_{\Gamma}$ with the same phases of $\Gamma$ and rank less than $k$. Moreover, we are free to pick the first row of $B_{\Gamma}$ to have modulus one, since scaling an entire column of $B_{\Gamma}$ by a positive integer does not change the rank or the phases.

Consider then $k \times m$ submatrices $\Gamma_i$ of $\Theta$, $i=1,\dots,\left\lfloor \frac{n-1}{k-1} \right\rfloor$, all containing the first row but otherwise pairwise disjoint. We can then construct a matrix $B$ by piecing together the $B_{\Gamma_i}$'s, since they coincide in the only row they share, and filling out the remaining rows, always less than $k-1$, with the corresponding entries of $\Theta$.

By construction, in that matrix $B$ we always have in the rows corresponding to $B_{\Gamma_i}$  a row different than the first that is a linear combination of the others, and can be erased without dropping the rank of $B$. Doing this for all $i$, we get that the rank of $B$ has at least a deficiency per $B_{\Gamma_i}$, so its rank is at most
$$n- \left\lfloor \frac{n-1}{k-1} \right\rfloor,$$
and since $B$ has the same phases as $\Theta$, $\rank_{\text{phase}}(\Theta)$ satisfies the intended inequality.
\end{proof}

It seems that the bound derived above is not very strong. Indeed, for a $100 \times 100$ matrix we can only guarantee that the phase rank is at most $91$, and for a $1000 \times 1000$ matrix that it is at most $929$. We can quantify a little how good this bound is. By inverting the function $\frac{2^{x-1}}{x}$ we get that $k^*$, the smallest $k$ that satisfies the inequality on Proposition \ref{prop:ub_phase} is given by
$$\left\lceil{-\frac{W_{-1}\left( - \frac{\ln 2}{2m} \right)}{\ln 2}} \right\rceil ,$$
where $W_{-1}$ is one of the branches of the Lambert $W$ function, the inverse of $f(x)=xe^x$. Using the bounds for $W_{-1}$ derived in \cite{chatzigeorgiou2013bounds}, one can conclude that $k^*$ satisfies the inequalities
$$\left\lceil \frac{1+\sqrt{2u}+u}{\ln{2}} \right\rceil \geq k^* \geq  \left\lceil \frac{1+\sqrt{2u}+\frac{2}{3}u}{\ln{2}} \right\rceil $$
where $u = \ln(m) + \ln(2)-\ln(\ln(2))-1 \approx \ln(m) - 0.05966$. This means that the best bound we can expect from  Proposition \ref{prop:ub_phase} for an $n \times n$ matrix grows as $n\left(1-\frac{c}{\log_2(n)}\right)$. Recall that for sign rank we have that the rank of an $n \times n$ matrix is at most $\frac{n}{2}(1+o(1))$ , so our bound is very far from a result of that type, that one would expect to be true. It is, however, a first step towards a more meaningful bound.

\section*{Conclusion}

In this paper we proposed the new notion of phase rank of a phase matrix. This is, in our view, a very natural object to study, that connects organically to a number of different problems that have been studied in the past. One of the main purposes of this text is to establish this new concept and survey these connections to areas like the sign rank, coamoeba theory or ray-nonsingularity. We set up the framework, providing the basic properties of the phase rank and surveying what previous works can tell us about it. Furthermore we make some contributions with new results, including the characterization of $3 \times 3$ matrices with phase rank $2$ in terms of colopsidedness and an upper bound of the phase rank.

\bibliographystyle{plain}
\bibliography{bibliografia}

\end{document}